\documentclass[10pt,a4paper]{article}
\usepackage[latin1]{inputenc}
\usepackage{anysize}
\usepackage{graphicx}
\marginsize{2.5cm}{2.5cm}{2cm}{1cm}
\usepackage{amssymb,amsmath,amscd,amsfonts,amsthm,mathrsfs}
\newcommand{\bb}[1]{\mathbb{#1}}
\newcommand{\cc}[1]{\mathcal{#1}}

\numberwithin{equation}{section}
\newtheorem{theorem}{Theorem}
\usepackage{url}
\begin{document}
\title{An $n$-in-a-row Game}
\author{Joshua Erde \thanks{DPMMS, University of Cambridge}}
\maketitle
\begin{abstract}
The usual $n$-in-a-row game is a positional game in which two player alternately claim points in $\bb{Z}^2$ with the winner being the first player to claim $n$ consecutive points in a line. We consider a variant of the game, suggested by Croft, where the number of points claimed increases by 1 each turn, and so on turn $t$ a player claims $t$ points. Croft asked how long it takes to win this game. We show that, perhaps surprisingly, the time needed to win this game is $(1-o(1))n$. 
\end{abstract}
\section{Introduction}
A \emph{positional game} is a pair $(X,\cc{F})$ where $X$ is a set and $\cc{F} \subset \bb{P}(X)$. We call $X$ the \emph{board}, and the members $F \in \cc{F}$ are \emph{winning sets}. The game is played by two players, Red and Blue, who alternately claim unclaimed points from the board. The first player to claim all points from a winning set wins. The \emph{$n$-in-a-row game} is a positional game played on $\bb{Z}^2$ where the winning sets are any $n$ consecutive points in a row, either horizontally, vertically or diagonally.  It is known \cite{B2008} that for $n \leq 4$ the $n$-in-a-row game is a first player win, and a draw for $n \geq 8$. It is believed that for $n=5$ the game is a first player win, and a draw for $n \geq 6$. \\
\ \\
In this note we consider a related game. Given a function $f:\bb{N} \rightarrow \bb{N}$ we define the \emph{$(n,f)$ game} to be a positional game played on the same board with the same winning sets as $n$-in-a-row, however now at time $t$ a player claims $f(t)$ points. The $n$-in-a-row game corresponds to the $(n,1)$ game, where $1$ is the constant function taking value $1$. In this note we will consider the $(n,\iota)$ game, where $\iota$ is the identity function. Unlike the $n$-in-a-row game the $(n,\iota)$ game is not (with perfect play) a draw, since at time $n$ some player will claim $n$ points and so can claim a winning set. Therefore either the first or second player must have a winning strategy. Croft \cite{C2011} asked the question, how long does it take for that player to win?\\
\ \\
Our aim is to show that neither player can win in time less than $(1-o(1))n$. In fact, we prove a stronger result by considering the Maker-Breaker version of the game. A \emph{Maker-Breaker game} is a pair $(X,\cc{F})$ where $X$ is a set and $\cc{F} \subset \bb{P}(X)$, and as before we call $X$ the board and the members $F \in \cc{F}$ winning sets. Two player, Maker and Breaker, alternately claim unclaimed points from the board, Maker colouring his points red and Breaker blue. If Maker is able to claim all points from a winning set he wins, otherwise Breaker wins. For more on Maker-Breaker games see Beck \cite{B2008}. The \emph{Maker-Breaker $(n,f)$ game} is played on the same board with the same winning sets as the $(n,f)$ game. It is obvious that the Maker-Breaker $(n,\iota)$ game is a Maker win. We will consider the question of how long it takes for Maker to win. More formally given a strategy $\Phi$ for Breaker and a strategy $\Psi$ for Maker, at some time $ T(\Phi,\Psi)_n$ Maker will first fully occupy a winning set. Note that trivially $ 2 \sqrt{n} \leq T(\Phi,\Psi)_n \leq n$. We let $T_n = \max_{\Phi} \min_{\Psi} T(\Phi,\Psi)_n$, that is, $T_n$ is the first time at which, with perfect play, Maker is guaranteed to have won.\\
\ \\
 In Section $2$ we provide a simple proof for a linear lower bound on $T_n$ which motivates the ideas for Section $3$ where we show that, perhaps surprisingly, Breaker has a strategy that gives $T_n \geq (1-o(1))n$. Of course, this strategy gives similar lower bounds for the ordinary (non Maker-Breaker) $(n,\iota)$ game when adopted by either player.

\section{A weak pairing strategy}
It is possible to show that the Maker-Breaker $(n,1)$ game is a Breaker win for $n \geq 8$ by utilising a pairing strategy. A \emph{pairing strategy} for a Maker-Breaker game $(X,\cc{F})$ is to find a family of distinct pairs $\{a_1,b_1\},\{a_2,b_2\}, \ldots \in X$, with $a_i \neq b_i$ for all $i$, such that any $F \in \cc{F}$ contains some pair $\{a_i,b_i\}$. Breaker's strategy is then: whenever Maker claims $a_i$ he claims $b_i$ and vice versa. For background on pairing strategies see Beck \cite{B2008}. \\
\ \\
A direct pairing strategy cannot be described for the Maker-Breaker $(n,\iota)$ game, since players claim more than one point at once. However in this section we are able utilise a similar idea to give a lower bound for $T_n$. Instead of pairing points, our plan is to assign to each point a direction, and have Breaker's strategy to be as follows: for each point that Maker claims, Breaker claims the next unclaimed point in that direction. If Maker wants to fully occupy a line, say from East to West, then he cannot claim too many points in it that have been assigned the directions East or West, or Breaker will claim a point inside the line. So we aim to find a way to assign directions to points such that each winning set will have approximately the right number of each direction in it.
\begin{theorem}
$T_n \geq \frac{2}{11}n - 6$.
\end{theorem}
\begin{proof}
We define a function $f:\bb{Z}^2 \rightarrow \{\text{N, NE, E, SE, S, SW, W, NW}\}$ such that:
\begin{itemize}
\item the points $(0,0)$ to $(10,0)$ are mapped to N, NE, E, SE, S, SW, W, NW, N, NE, and E respectively,
\item $f(x,y) = f(x+11,y)$ for all $x,y \in \bb{Z}$,
\item $f(x,y+1) = f(x-3,y)$ for all $x,y \in \bb{Z}$.
\end{itemize}  
So $f$ is periodic with period 11 on $\{(x,y)\,:\, x \in \bb{Z}\}$ for any $y$, and we shift the pattern by 3 to go from $(x,y)$ to $(x,y+1)$.\\
\ \\ 
{[}Here the number 11 was chosen since we want a function that is periodic, with the same period, in each direction. If a function is periodic horizontally on $\bb{Z}^2$ and shifts by $p$ to go from a row to the row above then it will clearly be periodic horizontally, vertically and diagonally, however it might have a smaller period. It is a simple check that to have the same period vertically it needs a period co-prime to $p$, and for the diagonals it needs a period co-prime to both $p-1$ and $p-2$. So for $p=3$ we need a period co-prime to $2,3,4$, but also larger than $8$, since each direction needs to appear at least once, and the smallest such number is $11$.{]}\\
\ \\
The important property of this function is that for any $x \in \{\text{N, NE, E, SE, S, SW, W, NW}\}$ and for any direction, if we look at 11 points in a row $v_1,\ldots ,v_{11}$ then $1 \leq |\{v_i\,:\,f(v_i) = x\}| \leq 2$, that is, the number of points assigned to each direction is between 1 and 2.  Breaker's strategy can now be described as follows: at time $2t$ Breaker looks at the $2t-1$ points Maker claimed on his turn $v_1, \ldots ,v_{2t-1}$ and for each $v_i$ claims the next available point in the direction $f(v_i)$ (any further points are claimed arbitrarily).\\
\ \\
Suppose that Maker wins at time $2t+1$. We consider the $n$ points in the winning line $L=\{v_1, \ldots ,v_n\}$ just before Breaker moves at time $2t$. Without loss of generality we will assume $L$ is in the East-West direction (the other cases are treated similarly). If Maker has claimed any point $v_i \in L$ such that $f(v_i)=$ E then Maker must also have claimed $v_j$ for all $j>i$ since otherwise Breaker will claim one of then at time $2t$. Similarly if Maker has claimed any point $v_i$ such that $f(v_i)=$ W then Maker must also have claimed $v_j$ for all $j<i$. Therefore if Maker has 3 points $\{v_{i_1}, v_{i_2}, v_{i_3}\}$, $i_1 < i_2 < i_3$, such that $f(v_{i_j})=$ E for all $j$, then $|i_3 - i_1| > 11$ and hence there is some point $v_k$, $i_1<k<i_3$, such that $f(v_k)=$W. By the preceding comment Maker must then already have claimed $v_j$ for all $j > i_1$ and also for all $j<k$ and hence Maker must already have claimed the whole line at time $2t-1$, contradicting our initial assumption.\\
\ \\
Therefore Maker can only claim at most $4$ of the points $v_i$ such that $f(v_i) \in \{\text{E,W}\}$ before time $2t+1$. Of the $n$ points in L at least $\frac{2}{11}(n-10) \geq \frac{2}{11}n - 2$ of the $v_i$ satisfy $f(v_i)\in \{\text{E,W}\}$, therefore at least $\frac{2}{11}n - 6$ points in $L$ must be unclaimed at time $2t+1$. So $T_n = 2t+1 \geq \frac{2}{11}n - 6$.
\end{proof}
\ \\
The constant $\frac{2}{11}$ could be improved by picking a larger prime instead of 11, and the same proof would show that $T_n \geq (\frac{1}{4}-\epsilon) n - C(\epsilon)$, where $C(\epsilon)$ is some constant depending on $\epsilon$. We also mention for interest that the same strategy also gives a proof that the usual Maker-Breaker $n$-in-a-row game is a Breaker win for sufficiently large $n$. In fact it shows that the Maker-Breaker $(n,k)$ game is a Breaker win for $n\geq \frac{11}{2}k + 30$ for any fixed $k$, although this is clearly not optimal in the case $k=1$.

\section{Minimization of lines}
In this section, for each $k \in [0,1)$, we give a strategy for Breaker such that Maker cannot win before time $kn$, for large enough $n$, and hence show that $T_n \geq (1-o(1))n$. When we consider the board position ``at time $t$" we mean just prior to the turn where $t$ points are claimed. We denote by $[x,y]$ the interval $\{ z \in \bb{Z} \,:\, x \leq z \leq y\}$ and by $[n]$ the interval $[1,n]$.\\
\ \\
Our plan is as follows: we will start by covering $\bb{Z}^2$ with a family $\cc{F}$ of straight lines of length $2n$ such that any winning set is contained in a member of our family and such that every $v \in \bb{Z}^2$ is in 8 lines. Given a line of length $2n$ with some points claimed but no winning set fully claimed, Breaker can place 2 more points inside that line such that no winning set can be fully claimed by Maker. So at time $2t$ Breaker can spoil $t$ of these lines. Indeed given such a line, without loss of generality $[2n]$, let $x$ be the largest point in $[1,n]$ which is either blue or unclaimed. Similarly let $y$ be the smallest point in $[n+1,2n]$ which is either blue or unclaimed. Note that, since Maker does not have a winning set in $[1,2n]$, it follows that $|x-y| < n$. After Breaker claims $x$ and $y$ then Maker can no longer fully claim a winning set in $[2n]$. After Breaker has played in such a way in $F \in \cc{F}$ we call $F$ \emph{bad}; otherwise $F$ is \emph{good}. We aim to show that if Breaker plays so as to minimise the largest number of points Maker has in any good $F$ then Maker cannot win at time $kn$ for any $0 \leq k < 1$.
\begin{theorem}
$T_n \geq (1-o(1))n$.
\end{theorem}
\begin{proof}
To define the family $\cc{F}$ we let
$$F_{i,j} = \{(x,i)\,:\,jn \leq x \leq (j+2)n-1\}$$
$$G_{i,j} = \{(i,y)\,:\,jn \leq y \leq (j+2)n-1\}$$
$$H_{i,j} = \{(i+k,k)\,:\,jn \leq k \leq (j+2)n-1\}$$
$$I_{i,j} = \{(i+k,-k)\,:\,jn \leq k \leq (j+2)n-1\}.$$
Then we let
$$\cc{F} = \bigcup_{i,j \in \bb{N}} F_{i,j} \cup \bigcup_{i,j \in \bb{N}} G_{i,j} \cup \bigcup_{i,j \in \bb{N}} H_{i,j} \cup \bigcup_{i,j \in \bb{N}} I_{i,j}. $$
Note that every point $v \in \bb{Z}^2$ is in 8 members of this family. For a given play of the game we define 
$$\cc{A}_r^t = \{F \in \cc{F}\,:\, F \text{ is good and the number of red points in } F \text{ at time } t \text{ is at least } r\}.$$
If Maker were to win at time $kn$ then we must have $|\cc{A}_{(1-k)n}^{kn}| \geq 1$. However if Breaker were to play to minimise the number of points Maker has in any good $F$ then 
$$|\cc{A}_{(1-k)n}^{kn-1}| \geq \frac{kn-1}{2} + 1 \geq \frac{kn}{4}$$
since otherwise Breaker will spoil all such sets. Now for a small $\epsilon>0$ we claim that
$$|\cc{A}_{(1-k - \epsilon)n}^{kn-2}| \geq \frac{kn}{4} - \frac{8k}{\epsilon}.$$
Indeed since each point is in  8 of the $F \in \cc{F}$ then by claiming $kn \geq kn-2$ points Maker can only claim $\epsilon n$ points in at most $8\frac{kn}{\epsilon n} = \frac{8k}{\epsilon}$ sets. Our aim is to iterate this argument to reach a contradiction that Maker must have claimed too many points. Breaker's strategy will be as follows: at time $2t$ he picks the $t$ good $F \in \cc{F}$ which have the most red points in them and spoils them (any further points are claimed arbitrarily). We define $ L_t = \max_{r} \{ |\cc{A}_r^t| \neq 0 \}$, and first show that $L_t$ can only grow linearly in $t$. Indeed if $L_{2t+1} = k$ then we see that $|\cc{A}_{k+17}^{2t+2}| \leq \frac{8}{17}(2t+1) \leq  t+1$, and so $L_{2t+3} \leq k +16$, since Breaker will spoil $t+1$ of the $F \in \cc{F}$ on his turn. Therefore for $t \leq \frac{n}{100}$, $L_{t} \leq \frac{16n}{100}$ and so Maker cannot win before time $\frac{n}{100}$.\\
\ \\
Now given $k \in (\frac{1}{100},1)$ we claim, for large enough $n$, $\cc{A}_{(1-k)n}^{kn} = \phi$ and so Maker cannot win at time $kn$. For ease of presentation we assume that $kn$ is an integer and odd and so Maker does indeed move at time $kn$. Suppose for a contradiction that $L_{kn} \geq (1-k)n$ and let $\epsilon = \frac{1}{\log{n}}$. We claim that for all $C\leq \frac{32}{k(1-k)}$

$$|\cc{A}_{(1-k-C\epsilon)n}^{kn-2C}| > C \frac{kn}{4} - C \frac{8k}{\epsilon}.$$

The claim holds for $C=0$, since $L_{kn} \geq (1-k)n$. If the claim holds for a given value of $C \leq \frac{32}{k(1-k)}$ then we must have

$$|\cc{A}_{(1-k-C\epsilon)n}^{kn-2C-1}| > (C+1) \frac{kn}{4} - C \frac{8k}{\epsilon}$$

since Breaker will spoil $\frac{kn-2C-1}{2} \geq \frac{kn}{4}$ of the $F \in \cc{F}$ with his turn. And so

$$|\cc{A}_{(1-k-(C+1)\epsilon)n}^{kn-2C-2}| > (C+1) \frac{kn}{4} - (C+1) \frac{8k}{\epsilon}$$

since if the inequality did not hold then at time $kn - 2(C+1)$ Maker must claim at least $\epsilon n$ points in at least $\frac{8k}{\epsilon}$ of the $F \in \cc{F}$, but this would require at least $\frac{8k}{\epsilon} \epsilon n \frac{1}{8} = kn$ points. We conclude that, with $C= \frac{32}{k(1-k)}$, 
$$|\cc{A}_{\frac{(1-k)n}{2}}^{kn-2C}| \geq |\cc{A}_{(1-k-\frac{C}{\log{n}})n}^{kn-2C}| > \frac{8n}{(1-k)} - \frac{32}{k(1-k)}8k\log{n} \geq \frac{4n}{(1-k)}.$$
But now to claim at least $\frac{(1-k)n}{2}$ red points in at least $\frac{4n}{(1-k)}$ of the $F \in \cc{F}$ requires at least $\frac{(1-k)n}{2}\frac{4n}{(1-k)}\frac{1}{8} = \frac{n^2}{4}$ points. However by time $kn-2C \leq kn$ Maker has claimed at most $\left( \frac{kn}{2} \right)^2 < \frac{n^2}{4}$ points.
\end{proof}
We mention that the same strategy would give $T_n \geq (1-o(1))n$ even if at time $2t$ Maker claimed $100t$ points and at time $2t+1$ Breaker claimed $0.001t$ points. Indeed whilst Breaker's strategy in Section 1 depended on the fact that $f(2t) \geq f(2t-1)$, this strategy only uses the fact that both $f(2t)$ and $f(2t+1)$ are linear in $t$. Hence same strategy will show that the $(n,f)$ game has $T_n \geq (1-o(1))n$ as long as $f(2t)= at+b$ and $f(2t+1)=ct+d$ for $a,b,c,d \in \bb{R}^+$. \\
\ \\
As we mentioned in the introduction, since the analysis of these strategies would be unchanged, up to a small constant, if Breaker were to play first, both of these strategies can be used by the losing player in the $(n,\iota)$ game and so the lower bounds on $T_n$ are also applicable to the $(n,\iota)$ game. Finally we mention an open problem. Since the $(n,\iota)$ game cannot end in a draw, either the first or second player will have a winning strategy, Croft also asked: \\
\ \\
{\bf Question 3: } Is the $(n,\iota)$ game a first or second player win?\\
\ \\
A small case analysis shows that player 1 wins for $n=1,3,4,6,7$ and player 2 wins for $n=2,5$. It is not clear if there is a simple formula that decides which player wins for general $n$.
\bibliography{n-in-a-row}
\bibliographystyle{plain}
\end{document}